\documentclass[12pt,a4paper]{amsart}
\usepackage{fullpage}   
\usepackage{hyperref} 
\hypersetup{colorlinks=true, urlcolor=blue}
\newtheoremstyle{plain} 
     {2ex}              
     {2ex}              
     {}         
     {}                 
     {\bfseries}        
     {}                 
     {1ex}              
     {\thmname{#1 }\thmnumber{#2}\thmnote{ \normalfont{(#3)}}.
}
\newtheoremstyle{remark}
     {2ex}              
     {2ex}              
     {}                 
     {}                 
     {\bfseries}        
     {}                 
     {1ex}              
     {\thmname{#1 }\thmnumber{#2}\thmnote{ \normalfont{(#3)}}.
}
\theoremstyle{plain}
\newtheorem{definition}{Definition}[section]

\newtheorem{remark}[definition]{Remark}
\newtheorem{proposition}[definition]{Proposition}
\newtheorem{lemma}[definition]{Lemma}
\newtheorem{theorem}[definition]{Theorem}

\newtheorem{example}[definition]{Example}

\numberwithin{equation}{section}
\newcommand{\NN}{\mathbb{N}}

\newcommand{\QQ}{\mathbb{Q}}
\newcommand{\RR}{\mathbb{R}}
\newcommand{\RRp}{\RR_+}
\newcommand{\CC}{\mathbb{C}}

\newcommand{\del}{\partial}
\newcommand{\dom}{\ensuremath{\mathrm{dom}}}

\newcommand{\mi}{\mathrm{i}}

\newcommand{\hl}[1]{\textnormal{\textbf{#1}}} 

\newcommand{\alphahat}{\ensuremath{\widehat{\alpha}}}
\newcommand{\Atilde}{\widetilde{A}}
\newcommand{\cl}[1]{\ensuremath{\overline{#1}}}

\newcommand{\grad}{\nabla}
\newcommand{\gradlog}{\grad_{\log}}

\newcommand{\Gcheck}{\check G}
\newcommand{\Gunder}{\underline{G}}
\newcommand{\holfunc}{\mathcal{O}}
\newcommand{\holgerm}{\mathcal{O}_p}
\newcommand{\htilde}{\ensuremath{\widetilde{h}}}

\newcommand{\jac}{J}
\newcommand{\jaclog}{\jac_{\log}}

\newcommand{\phat}{\ensuremath{\widehat{p}}}
\newcommand{\Phitilde}{\widetilde{\Phi}}
\newcommand{\Phitildeunder}{\underline{\Phitilde}}
\newcommand{\var}{\mathcal{V}}
\newcommand{\supp}{\operatorname{supp}}

\newcommand{\stirl}[2]{\genfrac{[}{]}{0pt}{}{#1}{#2}}

\newcommand{\Xtilde}{\ensuremath{\widetilde{X}}}
\begin{document}
\title[Improvements for multiple points]
  {Asymptotics of coefficients of \\ multivariate generating functions: \\
   improvements for multiple points}
\author{Alexander Raichev}
\address{Department of Computer Science \\ University of Auckland \\
  Private Bag 92019 \\ Auckland \\ New Zealand}
\email{raichev@cs.auckland.ac.nz}
\author{Mark C. Wilson}
\address{Department of Computer Science \\ University of Auckland \\
  Private Bag 92019 \\ Auckland \\ New Zealand}
\email{mcw@cs.auckland.ac.nz}
\begin{abstract}
Let $F(x)= \sum_{\nu\in\NN^d} F_\nu x^\nu$ be a multivariate power series with complex coefficients that converges in a neighborhood of the origin.
Assume $F=G/H$ for some functions $G$ and $H$ holomorphic in a neighborhood of the origin.
We derive asymptotics for the coefficients $F_{r\alpha}$ as $r \to \infty$ with $r\alpha \in \NN^d$ for $\alpha$ in a permissible subset of $d$-tuples of positive reals. More specifically, we give an algorithm for computing arbitrary terms of the asymptotic expansion for $F_{r\alpha}$ when the asymptotics are controlled by a transverse multiple point of the analytic variety $H = 0$. This improves upon earlier work  by R. Pemantle and M. C. Wilson.

We have implemented our algorithm in Sage and apply it to obtain accurate numerical results for 
several rational combinatorial generating functions. 

\end{abstract}
\date{\today}
\maketitle
\section{Introduction}\label{sec:intro}

In \cite{PeWi2002, PeWi2004} Pemantle and Wilson began a 
program of analytic combinatorics in several variables to derive asymptotic expansions of coefficients of combinatorial generating functions. 
In this article we continue that program by improving upon several of their results.

Let $F(x)= \sum_{\nu\in\NN^d} F_\nu x_1^{\nu_1}\cdots x_d^{\nu_d}$ be a complex power series with complex coefficients that converges in a neighborhood $\Omega$ of the origin.
Assume $F=G/H$ for some functions $G$ and $H$ holomorphic on $\Omega$.
For example, $F$ could be a rational function.
We derive asymptotics for the ``ray coefficients" $F_{r\alpha}$ as $r \to \infty$ with $r\alpha \in \NN^d$ for $\alpha$ in a permissible subset of $d$-tuples of positive reals.  

In \cite{PeWi2002}, Pemantle and Wilson derived the general form of the asymptotic expansion of $F_{r\alpha}$ for directions whose asymptotics are controlled by smooth points of the set $\var = \{x\in\Omega : H(x) = 0\}$ of singularities of $F$, that is, points where $\var$ is locally a complex manifold.
They gave an explicit formula for the leading term but no practical method for computing higher order terms.
In \cite{RaWi2008a}, a computational extension of \cite{PeWi2002}, we devised an algorithm and Maple implementation for computing these higher order terms.

In \cite{PeWi2004}, Pemantle and Wilson generalized \cite{PeWi2002} by deriving the form of the asymptotic expansion of $F_{r\alpha}$ for directions whose asymptotics are controlled by multiple points of $\var$, that is, points where $\var$ is locally a finite union of complex manifolds.
Again, they gave an explicit formula for the leading term but no practical method for computing higher order terms.
This article is a computational extension of \cite{PeWi2004} analogous to \cite{RaWi2008a}.
Herein we devise an algorithm and Sage~\cite{Sage} implementation for computing these higher order terms.

Why is it important in asymptotic analysis to have algorithms to compute higher order terms? 
There are several reasons.
First, the form of the algorithm or formulas itself can be insightful.
For example, the recasting of smooth point results in terms of Gaussian curvature in \cite{BBBP2011} yields a much clearer understanding, independent of coordinates, of how local geometry controls the asymptotic scale.
Second, computing higher order terms often gives one higher numerical accuracy at small values of $r$ than using the leading term alone.
Third, sometimes computing higher order terms is necessary.
For example, computing the variance of random variables via a generating function often requires third order asymptotics.
Fourth, computing higher order terms can be difficult or downright infeasible by hand ---indeed, this is usually the case with multivariate asymptotics.
An algorithm allows the end user to pass on the task to a computer.

\subsection*{Our Contribution}\label{ss:contribution}

In this article we give an algorithm for computing arbitrary terms of the asymptotic expansion for $F_{r\alpha}$ for directions whose asymptotics are controlled by a multiple point of $\var$ of order $n \ge 1$. 
We do this by first deriving an explicit formula in Section~\ref{sec:asymptotics_special} for the special case where $n \leq d$ and the ideal generated by the germ of $H$ in the ring of germs of holomorphic functions is radical. 
This generalizes the formula for the smooth point case $n = 1$ in \cite[Theorem 3.2]{RaWi2008a} and improves upon the formula in \cite[Theorem 3.5]{PeWi2004}, which gave an explicit formula for only the leading term.
We then show in Section~\ref{sec:asymptotics} how to reduce the general multiple point case to the special case.
This gives a unified method for the computation of higher-order asymptotics
that works for any value of $n$ and $d$. 
Our method of derivation uses Fourier-Laplace integrals as in \cite{PeWi2004}, but avoids the complications of infinite stationary phase sets.
We have implemented our algorithm in an open-source Sage file called \texttt{amgf.sage} that is downloadable from \href{http://www.alexraichev.org/research.html}{Raichev's website}, and in Section~\ref{sec:examples} we employ it to work out examples.
Section~\ref{sec:proofs} contains most of our proofs. 

To the best of our knowledge, the algorithms here and in \cite{RaWi2008a} are the first explicit, practical, and fairly general methods in the multivariate combinatorics literature for computing higher order asymptotic expansions.
\section{Preliminaries}\label{sec:preliminaries}

Throughout this article we make use of basic facts from local analytic geometry, a good reference for which is \cite{dePf2000}.

For brevity we write a power series $\sum_{\nu\in\NN^d} a_\nu (x_1 - p_1)^{\nu_1}\cdots (x_d - p_d)^{\nu_d}$ as $\sum_\nu a_\nu (x - p)^\nu$ and use the multi-index notation $\nu! = \nu_1!\cdots\nu_d!$, $r\nu = (r\nu_1, \ldots, r\nu_d)$, $\nu + 1 = (\nu_1 + 1, \ldots, \nu_d + 1)$, and $\del^\nu = \del_1^{\nu_1} \cdots \del_d^{\nu_d}$, where $\del_j$ is partial differentiation with respect to component $j$.

Let $\holfunc(\Omega)$ denote the $\CC$-algebra of holomorphic functions on an open set $\Omega\subseteq\CC^d$ and $\holgerm$ the $\CC$-algebra of germs of holomorphic functions at $p \in \CC^d$.
The latter algebra is a local Noetherian factorial ring whose unique maximal ideal is the set $\{f \in \holgerm : f(p)=0 \}$ of non-units.

We refer often to both $d$-tuples and $(d - 1)$-tuples and write $\hat a = (a_1, \ldots, a_{d - 1})$ given a tuple $a= (a_1, \ldots, a_d)$.
For simplicity we assume $d \ge 2$, though our formulas below also apply in the case $d = 1$ of univariate functions, after making the simple changes described in \cite[Remark 3.6]{RaWi2008a}. 

Let $\Omega\subseteq\CC^d$ be a neighborhood of the origin (an open subset of $\CC^d$ containing the origin) and $F(x) =\sum_{\nu} F_\nu x^\nu \in \holfunc(\Omega)$.
Assume $F = G/H$ for some relatively prime $G, H \in \holfunc(\Omega)$.
Let $\var$ be the set of singularities of $F$, namely the analytic variety $\{x \in \Omega: H(x)=0\}$ determined by $H$.
We will derive asymptotics for the ray coefficients $F_{r\alpha}$ as $r \to \infty$ with $r\alpha \in \NN^d$ for $\alpha$ in a permissible subset of $\RRp^d$, the set of $d$-tuples of positive reals.
For asymptotics of $F_\nu$ when $d = 2$ and $\nu \to \infty$ along more general paths see \cite{Llad2006}.

To begin we recall several key definitions from \cite{PeWi2002,PeWi2004}.

Just as in the univariate case, asymptotics for the coefficients of $F$ are determined by the location and type of singularities of $F$, that is, by the geometry of $\var$. 
Generally the singularities closest to the origin are the most important. 
We define `closest to the origin' in terms of polydiscs.
For $p \in \CC^d$, let $D(p) =\{x \in \CC^d : \forall j\; |x_j|\le|p_j|\}$ and $C(p) =\{x \in \CC^d : \forall j\; |x_j| = |p_j|\}$ be the respective polydisc and polycircle centered at the origin with polyradius determined by $p$.

\begin{definition}
We say that a point $p\in\var$ is \hl{minimal} if $\var\cap D(p)$ is contained in the boundary of $D(p)$, that is, if there is no point $x\in\var$ such that for all $j$, $|x_j| < |p_j|$.
We say that $p\in\var$ is \hl{strictly minimal} if $\var\cap D(p) = \{p\}$, and we say that $p$ is \hl{finitely minimal} if $\var\cap D(p)$ is finite.
\end{definition}

Note that $\var$ always contains minimal points.
To see this, let $p\in\var$ and define $f:\var\cap D(p)\to\RR$ by
$f(x)=\sqrt{x_1^2 +\cdots+ x_d^2}$. 
Since $f$ is a continuous function on a compact space, it has a minimum, and that minimum is a minimal point of $\var$.

The singularities of $F$ with the simplest geometry are the smooth/regular points of $\var$.
Asymptotics for $F_{r\alpha}$ dependent on smooth points were derived in \cite{PeWi2002,RaWi2008a}. 
Here we focus on asymptotics dependent on points with the next simplest geometry, that is, multiple points.

\begin{definition}\label{multiple}
Let $p\in\var$ and	consider the unique factorization of the germ of $H$ in $\holgerm$ into irreducible germs.  
Choosing representatives for these germs gives a factorization $H= H_1^{a_1} \cdots H_n^{a_n}$ valid in a neighborhood of $p$.
We say that $p$ is a \hl{multiple point of order $n$} if
\begin{itemize}
    \item for all $j$ we have $H_j(p) =0$, and
    \item every set of at most $d$ vectors from 
    $\{\grad H_1(p), \ldots, \grad H_n(p) \}$ is linearly independent.
\end{itemize}
We say that $p$ is a \hl{convenient multiple point of order $n$} if $p$ is a multiple point of order $n$ and there exists an index $k$ such that for all $j$ we have $p_k \del_k H_j(p) \neq 0$.
\end{definition}

In other words, $p$ is a multiple point of $\var$ iff $\var$ is locally a union of $n$ complex manifolds that intersect transversely at $p$
\footnote{In keeping with \cite{PeWi2008} we are simplifying matters by assuming transversality.  
For a more general definition of `multiple point' see \cite{PeWi2004}.}.
In particular, the multiple points of $\var$ of order $n=1$ are exactly the smooth points of $\var$, and so multiple points are generalizations of smooth points.
Notice also that the definition above depends only on information about $H$ in an arbitrarily small neighborhood of $p$ and so it is independent of the germ representatives chosen.
Lastly, to derive an asymptotic expansion of the coefficients $F_{r\alpha}$ we will need to consider the singularities of $F$ relevant to the direction $\alpha$. 
We call these singularities critical points, and they arise when approximating the Fourier-Laplace integrals we use to approximate $F_{r\alpha}$ (in Lemmas~\ref{Hormander} and \ref{critical_points}).
They also have a stratified Morse theoretic interpretation which, in the interest of simplicity, we will not pursue here; for more details see \cite[Section 3.1]{PeWi2008}.

\begin{definition}
Let $\alpha\in\RRp^d$ and let $p\in\var$ be a convenient multiple point and choose an index $k$ such that $p_k \del_k H_j(p)\neq 0$ for all $j=1, \ldots, n$.
Consider the scaled logarithmic gradient vectors
\[
\gamma_j(p) 
=
\left(
\frac{p_1 \del_1 H_j(p)}{p_k \del_k H_j(p)},\ldots,\frac{p_d \del_d H_j(p)}{p_k \del_k H_j(p)}
\right)
\]
for $j=1,\ldots, n$.
We say that $p$ is \hl{critical} for $\alpha$ if 
\[
\left( \frac{\alpha_1}{\alpha_k},\ldots,\frac{\alpha_d}{\alpha_k} \right)
= 
\sum_{j=1}^n s_j \gamma_j(p)
\]
for some $s_j \ge 0$, that is, if $\alpha$ lies in the conical hull of the $\gamma_j(p)$, which we call the \hl{critical cone} of $p$.
\end{definition}

\section{The full asymptotic expansion: special case}\label{sec:asymptotics_special}  

Let $p\in\var$ be a convenient multiple point of order $n$, and let $H= H_1^{a_1} \cdots H_n^{a_n}$ be a local factorization of $H$ about $p$ as above.

Without loss of generality and for concreteness and ease of notation, suppose $p_d \del_d H_j(p)\neq 0$ for all $j$.
Henceforth we breaking symmetry and base our explicit calculations on the index $d$.
For instance, when we talk about critical points, we divide by the index-$d$ terms $p_d \del_d H_j(p)$. 

\begin{remark}
For the remainder of this section we assume the special case of 
$a_1 =\ldots= a_n =1$ and $n\le d$.
\end{remark}

Now, to state our main results we need to define several auxiliary sets and functions, most of which are derived from $G$ and $H$ and arise from the integration tricks we use to approximate $F_{r\alpha}$.
The reader should feel free to skim over these definitions on a first reading, and move on to the main results starting at Theorem~\ref{asymptotics_soft}. 

We parametrize the $d$th coordinate in terms of the first $d-1$ coordinates.
Since $\del_d H_j(p)\neq 0$ for all $j$, we can apply the Weierstrass preparation theorem to each $H_j$ to get 
\[
  H_j(w,y) = U_j(w,y) \left( y - \frac{1}{h_j(w)} \right)
\]
in a neighborhood of $p$, where $U_j$ is holomorphic and nonzero at $p$, $h_j$ is holomorphic in a neighborhood of $\phat$ with $1/h_j(\phat) =p_d$, and  $\del_d H_j(w,1/h_j(w)) \neq 0$.
Thus
\begin{align*}
    H(w,y) 
    =& 
    U(w,y) \prod_{j=1}^n \left( y -\frac{1}{h_j(w)} \right)^{a_j} \\
    =& 
    U(w,y) \prod_{j=1}^n \left( \frac{-y}{h_j(w)} \right)^{a_j} 
	\prod_{j=1}^n \left( \frac{1}{y} -h_j(w) \right)^{a_j} \\
	=&
	U(w,y) \prod_{j=1}^n \frac{-y}{h_j(w)}  
	\prod_{j=1}^n \left( \frac{1}{y} -h_j(w) \right) 
	\quad\text{(since $a_1=\cdots=a_n=1$)}\\
\end{align*}
in a neighborhood of $p$, where $U=U_1\cdots U_n$.
We use reciprocals, because they turn out to be convenient for proving Lemma~\ref{critical_points} later on.

For $n\geq 2$ let
	$\Delta
    = 
    \{s\in\RR^{n-1} : s_j \ge 0 \text{ for all } j \text{ and } 
    \sum_{j=1}^{n-1} s_j \le 1 \}$,
the standard orthogonal simplex of dimension $n-1$.
This simplex comes from the residue calculation in Lemma~\ref{residue}.

Let $W$ be a neighborhood of $\phat$ on which the $h_j$ are defined. 
For $j= 0,\ldots, n-1$ and $\alpha\in\RRp^d$ define the functions $h:W\times\Delta\to\CC$, $A_j:\dom(U)\to\CC$, $e:[-1,1]^{d-1}\to\CC^{d-1}$, and $\Atilde_j,\htilde,\Phitilde:e^{-1}(W\cap C(\phat))\times\Delta\to\CC$ by
\begin{align*}
    \Gcheck(w,y)
    &=
    \frac{G(w,y)}{U(w,y)}\prod_{j=1}^n \frac{-h_j(w)}{y}\\
    h(w,s)  
    &= 
    s_1 h_1 + \cdots + s_{n-1} h_{n-1} + (1-\sum_{j=1}^{n-1} s_j)h_n \\
    A_j(w,y)    
    &= 
    (-1)^{n-1} y^{-n+j} \left( \frac{\del}{\del y} \right)^j \Gcheck(w,y^{-1}) \\
    e(t)    
    &= 
    (p_1 \exp(\mi t_1), \ldots, p_{d-1} \exp(\mi t_{d-1})) \\
    \htilde(t,s)
    &= 
    h(e(t),s) \\
    \Atilde_j(t,s)   
    &= 
    A_j(e(t),\htilde(t,s)) \\
    \Phitilde(t,s)
    &= 
    -\log (p_d \htilde(t,s)) +\mi\sum_{m=1}^{d-1} \frac{\alpha_m}{\alpha_d} t_m \\
\end{align*} 
Note that 
$F(w,y) = \Gcheck(w,y) / \prod_{j=1}^n (y^{-1} - h_j(w))$ and
that $\htilde$, $\Atilde_j$, and $\Phitilde$ are all $p^\infty$ functions.
The function $h$ comes from the residue calculation in Lemma~\ref{residue}, and the functions $e$$, \htilde$, $\Atilde_j$, and $\Phitilde$ come from the exponential change of variables in Lemma~\ref{FL_integral}. 

Let $\jaclog (H,p)$ denote the $n\times d$ logarithmic Jacobian matrix, the $j$th row of which is the logarithmic gradient vector $\gradlog H_j(p)=(p_1 \del_1 H_j(p),\ldots,p_d \del_d H_j(p))$.
Notice that if the convenient multiple point $p$ has all nonzero coordinates, then every subset $S \subseteq \{\gradlog H_1(p),\ldots,\gradlog H_n(p) \}$ spans a subspace of $\CC^d$ of dimension $|S|$.
Logarithmic gradients arise, essentially, from the exponential change of variables used to get a Fourier-Laplace integral in Lemma~\ref{FL_integral}.

If $\alpha$ is critical for $p$, then
\[
\alpha = \left( \frac{\alpha_d s_1^*}{p_d \del_d H_1(p)}, \ldots, 
\frac{\alpha_d s_n^*}{p_d \del_d H_n(p)} \right) \jaclog(p) 
\]
for some nonnegative tuple $s^*$ with $\sum_{j=1}^n s_j^* =1$.
Moreover, if $p$ has all nonzero coordinates, then the tuple $s^*$ is unique since $\jaclog(p)$ has rank $n\le d$. 
Let $\theta^*= (0,\ldots,0,s_1^*,\ldots,s_{n-1}^*) \in\RR^{d-1}\times\Delta \subset\RR^{d+n-2}$.

If the Hessian $\det \Phitilde''(\theta^*)$ is nonzero, then $p$ is called $\hl{nondegenerate}$ for $\alpha$. 
Critical points and nondegeneracy come into play in Lemmas~\ref{Hormander} and \ref{critical_points}.

\begin{remark}
In the smooth point case $n=1$ we can simplify the definitions above.  
In that case $H= H_1^{a_1}$ with $a_1=1$ (in this section) and we set
\begin{align*}
    h(w)
    &= 
    h_1(w) \\
    A_0(w)
    &=  
    y^{-1} \Gcheck(w,y^{-1}) \Big|_{y=h(w)} \\
    \Atilde_0(t) 
    &=
    A_0(e(t)) \\
    \htilde(t)
    &= 
    h(e(t)) \\
    \Phitilde(t) 
    &= 
    -\log(p_d \htilde(t) ) 
    +\mi\sum_{m=1}^{d-1} \frac{\alpha_m}{\alpha_d} t_m \\
    \theta^*
    &= 
    t^* = 0.
\end{align*}
\end{remark}

With the setup above we can now get to our main theorem.
It is an elaboration of the following formula that appeared in \cite{PeWi2004}.

\begin{theorem}\label{asymptotics_soft}
Let $\alpha\in\RRp^d$ and $p\in\var$ be a strictly minimal convenient multiple point with all nonzero coordinates that is critical and nondegenerate for $\alpha$.
Then there exist constants $b_q(\alpha)$ such that
\[
    F_{r\alpha} \sim 
    c^{-r\alpha } 
	\Bigg[
    (2\pi )^{(n-d)/2} (\det \Phitilde''(\theta^*))^{-1/2} 
	\sum_{q\ge 0}  b_q(\alpha) (r\alpha_d)^{(n-d)/2 - q}
    \Bigg]
\]
as $r \to \infty$ with $r\alpha \in \NN^d$.
Here $\Phitilde''(\theta^*)$ is the $(d + n - 2) \times (d + n - 2)$ Hessian matrix of $\Phitilde$ evaluated at $\theta^*$. 
\end{theorem}

\begin{proof}
Proved in \cite[Theorem 3.9]{PeWi2004}.
\end{proof}

We give an explicit formula for all the coefficients $b_q(\alpha)$.
Previously, only $b_0(\alpha)$ was known.

\begin{theorem}\label{asymptotics}
Let $\alpha\in\RRp^d$ and $p\in\var$ be a strictly minimal convenient multiple point with all nonzero coordinates that is critical and nondegenerate for $\alpha$.
Then

\begin{align*}
\label{asy_formula}
    F_{r\alpha} \tag{$\star$}
    =&
    c^{-r\alpha } \Bigg[
    (2\pi )^{(n-d)/2} (\det \Phitilde''(\theta^*))^{-1/2}
    \sum_{q=0}^{N-1}
 	(r\alpha_d)^{(n-d)/2 -q} \\
	& \times
    \sum_{\substack{0 \le j \le \min\{n-1,q\} \\ \max\{0, q-n\}  \le k \le q \\ 
	j+k \le q}}
    L_k(\Atilde_j, \Phitilde) \binom{n-1}{j} \stirl{n-j}{n+k-q} (-1)^{q-j-k}
      \\
    &
    + O\left((r\alpha_d)^{(n-d)/2-N}\right)
    \Bigg],
\end{align*}
as $r \to \infty$ and $r\alpha \in \NN^d$.

Here
\begin{align*}
    L_k(\Atilde_j,\Phitilde)
    &= 
	\sum_{l=0}^{2k} 
    \frac{\mathcal{H}^{k+l} (\Atilde_j \Phitildeunder^l)(\theta^*)}
    {(-1)^k 2^{k+l} l! (k+l)!}, \\
	\Phitildeunder(\theta)
	&= 
	\Phitilde(\theta)-\Phitilde(\theta^*)-\frac{1}{2} 
	(\theta-\theta^*) \Phitilde''(\theta^*) (\theta-\theta^*)^T,
\end{align*}
the differential operator $\mathcal{H}$ is given by
\[
	\mathcal{H} 
	= 
	-\sum_{1 \le a,b \le d+n-2} (\Phitilde''(\theta^*)^{-1})_{a,b} 
	\del_a\del_b,
\]
and $\stirl{a}{b}$ denotes the Stirling numbers of the first kind.
In every term of $L_k(\Atilde_j,\Phitilde)$ the total number of derivatives of $\Atilde_j$ and of $\Phitilde''$ is at most $2k+j$.

Moreover, for each positive integer $N$ the big-oh constant of (\ref{asy_formula}) stays bounded as $\alpha$ varies within a compact subset of $\RRp^d$ of the critical cone of $p$.
\end{theorem}

\begin{proof}
In the next section.
\end{proof}

\begin{remark}
	
In the smooth point case $n=1$, (\ref{asy_formula}) agrees with the formula in \cite[Theorem 3.2]{RaWi2008a}.
Moreover, in that case we can allow coordinates of $p$ to be zero as long as
$p_k \del_k H(p) \not= 0$ for some $k$.  
Also, when $n=1$ and $d=2$ we can drop the nondegeneracy hypothesis (\cite[Theorem 3.3]{RaWi2008a}).
\end{remark}

\begin{proposition}\label{Hessian}
Under the hypotheses of Theorem~\ref{asymptotics} we have
\[
    \Phitilde''(\theta^*)
    =
    \begin{pmatrix}
    A & -\mi C^T \\
    -\mi C & 0     
    \end{pmatrix},
\]
where $A$ is a $(d-1) \times (d-1)$ matrix, $C$ is an $(r - 1) \times (d - 1)$ real matrix, and
\begin{align*}
    A_{kl}
    &=
    \del_k \del_l \Phitilde (\theta^*) \\
    C_{kl} 
    &= 
    \frac{p_l \del_l H_k(p)}{p_d \del_d H_k(p)} - 
    \frac{p_l \del_l H_n(p)}{p_d \del_d H_n(p)}.
\end{align*}    
Notice that we only take derivatives with respect to $t$ in $A$.
\end{proposition}

\begin{proof}
Since $\Phitilde$ is $C^\infty$, its Hessian matrix is symmetric.
The formula for $A$ follows by definition.
To compute the remainder of the Hessian, let $s_n = \sum_{j<n} s_j$ for notational convenience.
For $l < d$ we have
\begin{align*}
    \frac{\del \Phitilde}{\del t_l}(0, s)
    &= 
    \frac{-\mi p_l \exp(\mi t_l) \sum_{j \le n} s_j \del_l h_j(e(t))}
    {\htilde(t, s)} + 
    \mi \frac{\alpha_l}{\alpha_d} \Big|_{(0,s)} \\
    &=
    -\mi p_d p_l \sum_{j \le n} s_j \del_l h_j(\phat) +
    \mi \frac{\alpha_l}{\alpha_d} \\
    \frac{\del \Phitilde}{\del s_l}(0, s)
    &=
    \frac{-h_l(e(t)) + h_n(e(t))}{\htilde(t, s)} \Big|_{(0,s)} \\
    &= 0.
\end{align*}
By the implicit function theorem we have
$\del_l h_j(w) = h_j(w)^2 \del_l H_j(w,1/h_j(w)) / \del_d H_j(w,1/h_j(w))$ for
$l<d$, $j \le n$, and $w\in W$.
So for $k, l < d$ we have
\begin{align*}
    \frac{\del^2 \Phitilde}{\del s_k \del t_l}(0, s)
    &= 
    -\mi p_d p_l (\del_l h_k(\phat) - \del_l h_n(\phat)) 
    \quad \text{(since $\del s_n /\del s_j = -1$)} \\
    &=
    -\mi \left( 
    \frac{p_l \del_l H_k(p)}{p_d \del_d H_k(p)} - 
    \frac{p_l \del_l H_n(p)}{p_d \del_d H_n(p)}
    \right) \\
    \frac{\del \Phitilde}{\del s_l}(0, s)
    &= 0,
\end{align*}    
as desired. 

Finally each $C_{kl} \in \RR$, because by \cite[Proposition 3.12]{PeWi2008} each $\frac{p_l \del_l H_k(p)}{p_d \del_d H_k(p)} \in \RR$.
\end{proof}

\begin{theorem}\label{n=d}
Under the hypotheses of Theorem~\ref{asymptotics}, when $n=d$ there exists $\epsilon \in (0,1)$ such that
\[
    F_{r\alpha}
    =
    p^{-r\alpha}\left[
    \frac{ \pm G(p)}{\det H'(p) \prod_{j\le d} p_j} +
    O(\epsilon^r) 
    \right]
\]    
as $r \to \infty$.
Here $H'(p)$ is the $n \times d$ Jacobian matrix of $H$ evaluated at $p$.

Moreover, the big-oh constant stays bounded as $\alpha$ varies within a compact subset of $\RRp^d$ of the critical cone of $p$.
\end{theorem}

\begin{proof}
By \cite[Corollary 3.24]{PeWi2008} all terms beyond the leading term in the asymptotic expansion of $p^{r\alpha} F_{r\alpha}$ are zero and the error term is exponentially decreasing.
(This follows from a Leray residue argument on the Cauchy integral of $F_{r\alpha}$.)

According to (\ref{asy_formula}) the leading term is 
$\frac{L_0(\Atilde_0, \Phitilde)}{\sqrt{\det(\Phitilde''(\theta^*))}}$.
First,
\begin{align*}
    L_0(\Atilde_0, \Phitilde)
    &=
    \mathcal{H^0}(\Atilde_0)(\theta^*) 
    =
    A_0(\phat, \frac{1}{p_d}) \\
    &=
    (-1)^{n-1} p_d^r \Gcheck(p) 
    =
    (-1)^{n-1} p_d^r \frac{G(p)}{U(p)} \prod_{j \le n} \frac{-1}{p_d^2} \\
    &=
    \frac{-G(p)}{U(p) p_d^r} 
    =
    \frac{-G(p)}{\prod_{j \le n} p_d \del_d H_j(p)}.
\end{align*}
Second, by Proposition~\ref{Hessian}, $\sqrt{\det(\Phitilde''(\theta^*))} = \sqrt{(\det C)^2} = |\det C|$ since $C$ is a real matrix.
Now consider the $n \times d$ matrix $\Gamma$ whose $j$th row is the scaled logarithmic gradient vector $\gamma_j(p) = \left(
\frac{p_1 \del_1 H_j(p)}{p_d \del_d H_j(p)}, \ldots, \frac{p_d \del_d H_j(p)}{p_d \del_d H_j(p)} \right)$.
Then
\begin{align*}
    \det C 
    =&
    \det 
    \begin{pmatrix}
    \gamma_1(p) - \gamma_n(p) \\
    \cdots \\
    \gamma_{n-1}(p) - \gamma_n(p) \\
    \gamma_n(p)    
    \end{pmatrix} \\
     & 
    \text{(by expanding the latter $n \times d$ matrix 
    by minors along its last column)} \\
    =&
    \det \Gamma \\
     &
    \text{(by similarity via elementary row operations)} \\ 
    =&
    \frac{\prod_{j \le n} p_j}{\prod_{j \le n}p_d \del_d H_j(p)} \det H'(p).
\end{align*} 
This proves the result.
\end{proof}

\section{Proving Theorem~\ref{asymptotics}}

To prove Theorem~\ref{asymptotics} we follow an approach similar to that of 
\cite{PeWi2002,PeWi2004,RaWi2008a}.
However, in contrast to those articles, here we first assume that $H$ has the relatively simple local factorization $H = H_1 \cdots H_n$ with $n \le d$ and then show in Section~\ref{sec:asymptotics} how to reduce to this case.
We take the following steps.
\begin{description}
\item[Step 1] Use Cauchy's integral formula to express $p^{r\alpha} F_{r\alpha}$ as a $d$-variate integral over a contour $C$ in $\Omega$. 
\item[Step 2] Expand the contour $C$ across $p_d$ and use Cauchy's residue theorem to  express the innermost integral as a residue.
\item[Step 3] Rewrite the residue as an $n$-variate integral over the simplex $\Delta$.
\item[Step 4] Rewrite the resulting integral as a Fourier-Laplace integral. 
\item[Step 5] Approximate the integral asymptotically.
\end{description}

Starting at step 1, we use Cauchy's integral formula to write
\[
    p^{r\alpha} F_{r\alpha} 
    =
 	p^{r\alpha} \frac{\del^{r\alpha} F(0)}{(r\alpha!)}
	=
    p^{r\alpha} \frac{1}{(2\pi\mi)^d} \int_C \frac{G(w) dw}{w^{r\alpha+1} H(w)},
\] 
where $C$ is a contour in $\Omega$.
We then follow steps 2--5 by applying the following lemmas, the proofs of which have been swept away to Section~\ref{sec:proofs} to clarify the logical flow of the main argument.

\begin{lemma}[for step 2]\label{integral}  
Let $\alpha\in\RRp^d$ and $p\in \var$ be a strictly minimal convenient multiple point with nonzero coordinates. 
There exists $\epsilon\in(0,1)$ and a polydisc neighborhood $D$ of $\phat$ such that
\[
    p^{r\alpha} F_{r\alpha}  
    =  
    p^{r\alpha} (2\pi\mi)^{1-d} \int_{X} \frac{-R_r(w)}{w^{r \hat{\alpha} + 1}}  dw
    + O\left(\epsilon^r\right)
\]
as $r \to \infty$ with $r\alpha \in \NN^d$, where 
$X=D\cap C(\phat)$ and $R_r(w)$ is the sum over $j$ of the residues of 
$y \mapsto y^{-r\alpha_d-1} F(w,y)$ at $h_j(w)$.
\end{lemma}

\begin{proof}
Proved in \cite[proof of Lemma 4.1]{PeWi2002}.
\end{proof}	
	
\begin{lemma}[for step 3]\label{residue}
In the previous lemma for $n\ge 2$ we have
\[
    R_r(w)
    = 
    \int_\Delta  
    \left( \frac{\del}{\del y} \right)^{n-1} 
    (-1)^{n-1} f_r(w,y) \Big|_{y=h(w,s)} ds,
\]
where $f_r(w,y) = -y^{r\alpha_d -1} \Gcheck(w,y^{-1})$ and $ds$ is the standard volume form $ds_1 \wedge \cdots \wedge ds_{n-1}$.
For the smooth case $n = 1$ we have $R_r(w)= f_r(w,h(w))$.
\end{lemma}

\begin{proof}
See Section~\ref{sec:proofs}.
\end{proof}	

For $j=0,\ldots,n-1$ define $P_j:\NN\to\NN$ by $P_j(r) = \binom{n-1}{j} (r\alpha_d - 1)^{\underline{n-1-j}}$.
The falling factorial powers in $P_j$ are defined by $a^{\underline{k}} = a(a-1)\cdots(a-k+1)$ and $a^{\underline{0}}=1$ for $a\in\RR$ and $k\in\NN$.
So the degree of $P_j$ in $r$ is $n-1-j$.

\begin{lemma}[for step 4]\label{FL_integral}
For $n \ge 2$,
\[
    p^{r\alpha} F_{r\alpha}
    =
    (2\pi)^{1-d} \sum_{j=0}^{n-1} P_j(r) 
    \int_{\Xtilde} \int_\Delta  
    \Atilde_j(t,s) \exp(-r\alpha_d \Phitilde(t,s)) ds \, dt 
    + O(\epsilon^r),
\]
as $r \to \infty$ with $r\alpha \in \NN^d$, where $\Xtilde = e^{-1}(X)$.
For $n=1$, 
\[
    p^{r\alpha} F_{r\alpha}
    =
    (2\pi)^{1-d}  \int_{\Xtilde}   
    \Atilde_j(t) \exp(-r\alpha_d \Phitilde(t)) \; dt 
    + O(\epsilon^r),
\]
as $r \to \infty$ with $r\alpha \in \NN^d$.
\end{lemma}

\begin{proof}
See Section~\ref{sec:proofs}.
\end{proof}	

The next lemma on Fourier-Laplace integrals provides our key approximation. 
The function spaces mentioned are complex valued.
A stationary and nondegenerate point of a function $g$ is a point $\theta^*$ such that $\grad g(\theta^*)=0$ and $\det g''(\theta^*) \neq 0$, respectively.

\begin{lemma}[for step 5]\label{Hormander}
Let $\mathcal{E}\subset \RR^{m}$ be open, $N$ a positive integer, and $q = N+\lceil m/2 \rceil$.
If $A \in C^{2q}(\mathcal{E})$ with compact support in $\mathcal{E}$, $\Phi\in C^{3q+1}(\mathcal{E})$, $\Re\Phi\ge 0$, $\Re \Phi(\theta^*)=0$, $\Phi$ has a unique stationary point $\theta^* \in\supp A$, and $\theta^*$ is nondegenerate, 
then  
\[
\int_{\mathcal{E}} A(\theta) \exp(-\omega \Phi(\theta)) d\theta
=
\exp(-\omega \Phi(\theta^*)) 
(\det\left(\frac{\omega \Phi''(\theta^*)}{2\pi} \right))^{-1/2}
\sum_{k=0}^{N-1} 
\omega^{-k} L_k(A,\Phi)
+O\left( \omega^{-m/2-N} \right),
\]
as $\omega\to\infty$.

Here $L_k$ is the function defined in Theorem~\ref{asymptotics} with 
$m=d+n-2$.
Moreover, the big-oh constant is bounded when the partial derivatives of $\Phi$ 
up to order $3q +1$ and the partial derivatives of $A$ up to order $2q$ all stay bounded in supremum norm over $\mathcal{E}$.
\end{lemma}

\begin{proof}
Proved in \cite[Theorem 7.7.5]{Horm1983}.
\end{proof}	

The final lemma ensures that the hypotheses of Lemma~\ref{Hormander} are 
satisfied in our setting.

\begin{lemma}[for step 5]\label{critical_points}
Let $\alpha\in\RRp^d$ and $p$ be a strictly minimal convenient multiple point that is critical and nondegenerate for $\alpha$. 
Then on $\Xtilde\times\Delta$, we have $\Re\Phitilde \geq 0$ with equality only at points of the form $(0,s)$ (and only at zero for $n=1$), and $\Phitilde$ has a unique stationary point at $\theta^*$.
\end{lemma}

\begin{proof}
See Section~\ref{sec:proofs}.
\end{proof}	

We can now prove Theorem~\ref{asymptotics}.

\begin{proof}[Proof of Theorem~\ref{asymptotics}]
By Lemmas \ref{integral} and \ref{FL_integral} there exists $\epsilon\in(0,1)$ 
and an open bounded neighbourhood $\Xtilde$ of $0$ such that
\[
    p^{r\alpha}F_{r\alpha}
    = 
    (2\pi)^{1-d} \sum_{j=0}^{n-1} P_j(r) I_{j,r} 
    +O\left(\epsilon^r\right) 
\]
as $r \to \infty$ with $r\alpha \in \NN^d$, where 
$I_{j,rn} = \int_{\mathcal{E}}   
\Atilde_j(\theta) \exp(-r\alpha_d \Phitilde(\theta)) d\theta$
and $\mathcal{E} = \Xtilde\times\Delta^\circ$, where $\Delta^\circ$ is the interior of $\Delta$.

Choose $\kappa \in C^\infty(\mathcal{E})$ with compact support in $\mathcal{E}$ (a bump function) such that $\kappa = 1$ on a neighbourhood $Y$ of $\theta^*$.
Then
\[
    I_{j,r} 
    = 
    \int_{\mathcal{E}} \kappa(\theta) \Atilde_j(\theta) 
    \exp(-r\alpha_d \Phitilde(\theta)) d\theta
    +\int_{\mathcal{E}} (1-\kappa(\theta)) \Atilde_j(\theta) 
    \exp(-r\alpha_d\Phitilde(\theta)) d\theta.
\] 
The second integral decreases exponentially as $r \to \infty$ since $\Re\Phitilde$ 
is strictly positive on the compact set $\cl{\mathcal{E} \setminus Y}$ by 
Lemma~\ref{critical_points}.
By Lemma~\ref{critical_points} again and the nondegeneracy hypothesis, 
we we may apply Lemma~\ref{Hormander} to the first integral.
Noting that $L_k(\kappa\Atilde_j, \Phitilde) = L_k(\Atilde_j, \Phitilde)$ because the derivatives are evaluated at $\theta^*$ and $\kappa=1$ in a neighborhood of $\theta^*$, we get 
\begin{align*}
    I_{j,r}
    &=
    \exp(-n_d \Phitilde(\theta^*))
    (\det\left(\frac{r\alpha_d \Phitilde''(\theta^*)}{2\pi}\right))^{-1/2}
    \sum_{k=0}^{N-1} 
    (r\alpha_d)^{-k} L_k (\Atilde_j, \Phitilde)
    + O( (r\alpha_d)^{-(d-1+n-1)/2-N}) \\
    &=
    (2\pi)^{(d+n-2)/2} (\det\Phitilde''(\theta^*))^{-1/2}
    \sum_{k=0}^{N-1} L_k(\Atilde_j, \Phitilde) (r\alpha_d)^{-(d+n-2)/2-k} 
    + O\left( (r\alpha_d)^{-(d+n-2)/2-N} \right)
\end{align*}
as $r \to \infty$ with $r\alpha \in \NN^d$.

Notice that for $j=0,\ldots,n-1$ each $I_{j,r}$ has error $O( (r\alpha_d)^{-(d+n-2)/2-N})$ and each $P_j(r)$ has degree $r-j-1$ in $n$.
Thus the error in the asymptotic expansion for $p_{r\alpha}F_{r\alpha}$ will be a sum of terms of the form $O( (r\alpha_d)^{(n-d)/2-N-j})$ which is $O( (r\alpha_d)^{(n-d)/2-N})$.
So 
\begin{align*}
p^{r\alpha}F_{r\alpha}
&=
(2\pi)^{(n-d)/2} (\det\Phitilde''(\theta^*))^{-1/2}
\sum_{q=0}^{N-1} b_q(\alpha) (r\alpha_d)^{(n-d)/2 -q} 
+ O\left( (r\alpha_d)^{(n-d)/2-N} \right) \\
&=
(2\pi)^{(n-d)/2} (\det\Phitilde''(\theta^*))^{-1/2}
\sum_{j=0}^{n-1} \sum_{k=0}^{N-1} P_j(r) L_k(\Atilde_j, \Phitilde) (r\alpha_d)^{-(d+n-2)/2-k} 
+ O\left( (r\alpha_d)^{-(n-d)/2-N} \right).
\end{align*}
Let us expand $P_j(r)$ and collect like powers to find the coefficients $b_q(\alpha)$.

The falling factorial powers satisfy $(a-1)^{\underline{m}} = (a-1) \dots (a-1-k) = \frac{1}{a} a^{\underline{m+1}}$ and are related to regular powers and Stirling numbers of the first kind via 
\[
a^{\underline{m}} = \sum_{l=0}^m \stirl{m}{l} (-1)^{m-l} a^l; 
\]
see \cite[(6.13)]{GKP1994} for instance.
Thus
\[
P_j(r) 
= 
\binom{n-1}{j} \frac{1}{r\alpha_d} 
\sum_{l=0}^{n-j} \stirl{n-j}{l} (-1)^{n-j-l} (r\alpha_d)^l,
\]
and so
\[
\sum_{q=0}^{N-1} b_q(\alpha) (r\alpha_d)^{(n-d)/2 - q} 
= 
\sum_{j=0}^{n-1} \sum_{k=0}^{N-1} \sum_{l=0}^{n-j} 
L_k(\Atilde_j, \Phitilde) \binom{n-1}{j} \stirl{n-j}{l} (-1)^{n-j-l}  
(r\alpha_d)^{-(d+n)/2 -k +l}.  
\]
The coefficient $b_q(\alpha)$ is found by imposing the constraint $(n - d)/2 - q = -(d + n)/2 - k + l$. 
Thus $l = n + k - q$, and we can eliminate the $l$-sum to arrive at formula (\ref{asy_formula}).

Lastly, regarding uniformity, we may assume that the $\Atilde_j$ and $\Phitilde$ are defined and hence $C^\infty$ on a neighborhood of the closure of $\mathcal{E}$, so that their derivatives up to any given order all stay bounded in supremum norm over $\mathcal{E}$.  
Now suppose $\alpha$ varies within a compact subset $K\subset\RRp^d$ of the critical cone of $p$.
Since $\jaclog(H,p)$ has rank $n \le d$ it is a bijective linear transformation from $\RR^n$ to its image in $\RR^d$ and therefore a bicontinuous function.
Thus its inverse maps $K$ to a compact set $K'$ of $\theta^*$s in $\mathcal{E}$.
Choose the neighborhood $Y$ in the argument above to contain $K'$ so that one bump function $\kappa$ works for all $\theta^*$.
Since the derivatives of the $\kappa\Atilde_j$ and $\Phitilde$ up to any given order all stay bounded in supremum norm over $\mathcal{E}$ and since only $\Phitilde$ and $\Phitilde'$ depend on $\alpha$ but continuously, we conclude by Lemma~\ref{Hormander} that for any given $N$, the big-oh constant in (\ref{asy_formula}) remains bounded as $\alpha$ varies within $K$.
\end{proof}

\section{The full asymptotic expansion: general case}\label{sec:asymptotics}

Again let $p\in\var$ be a strictly minimal convenient multiple point of order $n$ with all coordinates nonzero and let $H= H_1^{a_1} \cdots H_n^{a_n}$ be a local  factorization of $H$.
We deal now with the case of arbitrary $a_j$ and $n$.

In step 2 of the previous section the Cauchy integral can be manipulated to reduce to the special case $a_1 = \ldots = a_n = 1$ and $n \le d$.
More specifically, we amend our plan by inserting these three steps after step 2:

\begin{enumerate}
\item[(2a)] 
If $r>d$, then decompose $F$ as a sum of fractions whose denominators are of  type $\prod_{j\in J} H_j^{b_j}$ where $J$ is a size $d$ subset of $\{1, \ldots, n\}$ and each $b_j$ is an integer with $b_j\le a_j$.
So each denominator in the sum has only $d$ irreducible factors of $H$.

\item[(2b)] If some irreducible factor of $H$ is repeated, then treat each resulting integral as the integral of a holomorphic form, and rewrite each integral as the sum of integrals whose denominators are of type $w^{r\alpha+1} \prod_{j\in J} H_j$ where $J$ is a size at most $d$ subset of $\{1, \ldots, n\}$.
So each holomorphic form has a denominator with at most $d$ unrepeated irreducible factors of $H$.

\item[(6)] Add up all the asymptotic expansions.
\end{enumerate}

The following two lemmas prove that these additional steps are possible.

\begin{lemma}[for step 2a]\label{reduce-factors}
Let $p$ be a multiple point of $H := H_1^{a_1} \cdots H_n^{a_n}$, where $n > d$, each $H_j$ is holomorphic in a neighborhood $U$ of $p$, and the germ of each $H_j$ is prime.
Then for any function $G$ holomorphic on $U$, there exists a neighborhood of $p$ in which we have the partial fraction decomposition
\[
    \frac{G}{H} 
    = 
    \sum_J \frac{G_J}{\prod_{j\in J} H_j^{b_j}}, 
\]
where each $G_J$ is holomorphic (and possibly zero), $J$ ranges over all subsets of $\{1, \ldots, n\}$ of size $d$, and for each $J$ we have $\sum_{j\in J} b_j = \sum_{i=1}^n a_i$.
\end{lemma}

\begin{proof}
Since $p$ is a multiple point of $H$, the gradients at $p$ of any $d$ of the $H_j$ are linearly independent.
Thus the germs of any $d$ of the $H_j$ generate the maximal ideal in $\holgerm$ by \cite[Corollary 5.4]{Ruiz1993}.
In particular, the germ of $H_1$ is in the ideal of the germs of $H_2, \ldots, H_n$, and so in a neighborhood of $p$ we have
\[
    H_1 = \sum_{j=2}^n g_j H_j
\]
for some holomorphic functions $g_j$.
Therefore, in that neighborhood we have
\begin{align*}
\frac{G}{H}
=&
\frac{G \sum_{j=2}^n g_j H_j}{H_1^{a_1 + 1} H_2^{a_2} \cdots H_n^{a_n}} \\
=&
\frac{G_2}{H_1^{a_1 + 1}H_2^{a_2 - 1} H_3^{a_3} \cdots H_n^{a_n}} +
\frac{G_3}{H_1^{a_1 + 1}H_2^{a_2} H_3^{a_3 - 1} H_4^{a_4} \cdots H_n^{a_n}} + \ldots + \\
 &
\frac{G_n}{H_1^{a_1 + 1}H_2^{a_2} \cdots H_{n-1}^{a_n} H_n^{a_n - 1}},
\end{align*}
where $G_j = G g_j$.
Notice that in the denominator of each resulting summand, the sum of the degrees of all the $H_j$ remains $\sum_{i=1}^n a_i$. 

Recursively repeating this procedure on each summand (always singling out $H_1$, say) yields the desired result in finitely many steps.
\end{proof}

\begin{remark}
When the $H_j$ are polynomials from a computable polynomial ring, such as $\QQ[x]$, the procedure in the proof above is computable.
Alternatively a partial fraction expansion can be computed according to the algorithm in \cite{Lein1978}, which is not applicable to the analytic case.
\end{remark}

\begin{lemma}[for step 2b]\label{reduce-powers}
Let $p$ be a multiple point of $H := H_1^{b_1} \cdots H_n^{b_n}$, where $n \le d$, each $H_j$ is holomorphic in a neighborhood $U$ of $p$, and the germ of each $H_j$ is prime, and let $\var_j := \{x \in U: H_j(x) = 0\}$.
Then for any function $\Gunder$ holomorphic on $U$, there exists a neighborhood $U'$ of $p$ such that the holomorphic form
\[
    \frac{\Gunder(x)}{H(x)} dx
\]
is de Rham cohomologous in $U'\setminus(\var_1 \cup\cdots\cup \var_n)$ to a holomorphic form
\[
    \sum_J \frac{\Gunder_J(x) dx}{\prod_{j\in J} H_j(x)}, 
\]
where each $\Gunder_J$ is holomorphic (and possibly zero) on $U'$ and $J$ ranges over all subsets of $\{1, \ldots, n\}$.
In particular, the integrals of the two forms above over a polycircle in $U'\setminus(\var_1\cup\cdots\cup\var_n)$ are equal.
\end{lemma}

\begin{proof}
Proved in \cite[Theorem 17.6]{AiYu1983}.
\end{proof}

\begin{remark}
When the $H_j$ are polynomials from a computable polynomial ring, such as $\QQ[x]$, the procedure in the proof above is computable.

When applying Lemma~\ref{reduce-powers} in step 2b to our local integrals of residues, $\Gunder(x)$ will be of the form $G(x)/x^{r\alpha+1}$ where $G(x)$ does not contain $r$ and for each $J$ we will have $\sum_{j\in J} b_j = \sum_{i=1}^n a_i$.
Thus upon inspection of the constructive proof of Lemma~\ref{reduce-powers}, the cohomologous form will have $r$-degree at most $\sum_{i=1}^n (a_i - 1)$, where the powers of $n$ arise from the derivatives of $G(x)/x^{r\alpha+1}$.

In particular, if $n \ge d$ and the other assumptions of Theorem~\ref{n=d} hold, then we can combine Lemmas~\ref{reduce-factors} and \ref{reduce-powers} and Theorem~\ref{n=d} to conclude that the leading term of the asymptotic expansion of $p^{r\alpha}F_{r\alpha}$ is a polynomial of degree at most $\sum_{i=1}^n a_i - n$, as is also shown in \cite[Theorem 3.6]{PeWi2004}.
\end{remark}

\begin{remark}[for step 6]
When computing the asymptotics for $G/H$ in a direction $\alpha$ by summing up the asymptotic contributions from the terms of the form $G_J / \prod_{j\in J} H_j$ where $J$ has size at most $d$, the only terms that will contribute to the expansion (modulo an exponentially decreasing error term) are the ones whose   critical cone (the conical hull of $\{\gamma_j(p): j \in J\}$) contains $\alpha$ and whose numerator does not vanish at $p$ \cite[Section 5]{Pema2000}.
In the case where all such contributing terms have numerators that vanish at $p$, a finer analysis is required to determine the correct asymptotics of $G/H$ which we do not provide here (but will be included in Pemantle and Wilson's forthcoming book on analytic combinatorics in several variables).
\end{remark}
    
\begin{remark}
In case $p$ is finitely minimal, for each point $x$ of $\var\cap C(p)$ we 
simply find an open set around $x$ and apply the general procedure above.
After that we sum the resulting asymptotic expansions over the finitely many 
$x$. 
\end{remark}
\section{Examples}\label{sec:examples}

Let us apply the formulas and procedures of Sections~\ref{sec:asymptotics_special} and \ref{sec:asymptotics} to a few combinatorial examples, that is, to functions with all nonnegative Maclaurin coefficients.
We will use our Sage package \texttt{amgf.sage}.

We focus on combinatorial examples $F(x)$, because for any $\alpha\in\RRp^d$ there is a minimal point in $\var \cap \RRp^d$ that determines the asymptotics for $F_{r\alpha}$ (\cite[Theorem 3.16]{PeWi2008}). 

Since there is no known computable procedure to factor an arbitrary polynomial $H$ in the analytic local ring of germs of holomorphic functions about $p$, we choose examples where $H$ is a polynomial whose local factorization in the algebraic local ring about $p$ equals its factorization in the analytic local ring about $p$, that is, $H$ is a polynomial whose irreducible factors in $\CC[x]$ are all smooth at $p$.

%
%
%

\begin{example}[$n < d$, no repeated factors]
Consider the trivariate rational function
\[
F(x,y,z)
=
\frac{1}{(1-x(1+y))(1-zx^2(1+2y))}
\]
in a  neighborhood $\Omega$ of the origin; cf \cite[Example 4.10]{PeWi2008}.
Its coefficients $F_\nu$ are all nonnegative, and its denominator 
$H(x,y,z)$ factors over $\CC[x,y,z]$ into irreducible terms $H_1(x,y,z)= 1-x(1+y)$ and $H_2(x,y,z)= 1-zx^2(1+2y)$, both of which are globally smooth.

The set of non-smooth/singular points of $\var= \{(x,y,z)\in\Omega : H(x,y,z) =0 \}$ is $\var' = \{ (x,y,z)\in\Omega : H(x,y,z)=\grad H(x,y,z) =0\} = 
\{ (1/(a+1),a,(a+1)^2/(2a+1)): a\in\CC\setminus\{-1\} \}$, which consists entirely of convenient multiple points of order $n = 2$.
They are not convenient in coordinate $d = 3$, but are in coordinate $d - 1= 2$, which we use here for our calculations.
A simple check shows that the points $(1/(a+1),a,(a+1)^2/(2a+1))$ for $a>0$ are strictly minimal.

The critical cone for each such point is the conical hull of the vectors $\gamma_1= (1,a/(a+1),0)$ and $\gamma_2= (1,a/(2a+1),1/2)$.

For instance, $p= (1/2,1,4/3)$ controls asymptotics for all $\alpha$ in the conical hull of the vectors $\gamma_1(p)= (2,1,0)$ and $\gamma_2(p)= (3,1,3/2)$.
For instance, $\alpha= (8,3,3)$ is in this critical cone, and applying Theorem~\ref{asymptotics} we get
\[
F_{r\alpha} 
= 
108^r \left[
\frac{3}{\sqrt{21\pi}}\, r^{-1/2} 
-\frac{1231}{8232\sqrt{21\pi}}\, r^{-3/2}
+\frac{329047}{58084992\sqrt{21\pi}}\, r^{-5/2}
+ O(r^{-7/2})
\right]
\]
as $r \to \infty$.

Calling the one-term, two-term, and three-term truncations of this asymptotic formula $S_1(r)$, $S_2(r)$, and $S_3(r)$, respectively and comparing them with the actual values of $F_{r\alpha}$ for small $r$, we get the following table.

\begin{table}[htbp]
\footnotesize{
\begin{tabular}{|l|llll|}
\hline 
$r$ & 1 & 2 & 4 & 8 \\
\hline
$108^{-r} F_{r\alpha}$ & 
	0.3518518519 & 0.2548010974 & 0.1823964231 & 0.1297748629 \\
$108^{-r} S_1(r)$ &
 	0.3693487820 & 0.2611690282 & 0.1846743909 & 0.1305845142 \\
$108^{-r} S_2(r)$ &
	0.3509381749 & 0.2546598957 & 0.1823730650 & 0.1297708726 \\
$108^{-r} S_3(r)$ &
 	0.3516356189 &	0.2547831876 & 0.1823948602 & 0.1297747255 \\
$108^{-r} S_1(r)$ rel err &
	-0.04972811712 & -0.02499177148 & -0.01248910347 & -0.006238891584 \\
$108^{-r} S_2(r)$ rel err &
	0.002596766210 & 0.0005541644108 & 0.0001280622701 & 0.00003074786527 \\
$108^{-r} S_3(r)$ rel err &
  	0.0006145569473 & 0.00007028933620 & 0.000008568698736 & 0.000001058756657 \\
\hline
\end{tabular}
}
\vspace{1em}
\caption{Successive approximations to $p^{-r\alpha} F_{r\alpha}$ with relative errors for $\alpha= (8,3,3)$.}
\end{table}
\end{example}

\begin{example}[$n < d$, no repeated factors]
Consider the trivariate rational function 
\[
	F(x,y,z) = \frac{16}{(4-2x-y-z)(4-x-2y-z)}
\]
in a  neighborhood $\Omega$ of the origin; cf \cite[Example 3.10]{PeWi2004}.
Its coefficients $F_\nu$ are all nonnegative, and its denominator 
$H(x,y,z)$ factors over $\CC[x,y,z]$ into irreducible terms $H_1(x,y,z)= 4-2x-y-z$ and $H_2(x,y,z)= 4-x-2y-z$, both of which are globally smooth.

The set of non-smooth points of $\var= \{(x,y,z)\in\Omega : H(x,y,z) =0 \}$ is $\var' = \{ (x,y,z)\in\Omega : H(x,y,z) =\grad H(x,y,z) =0\} = 
\{ (1-a,1-a,1+3a : a\in\CC\}$, which contains a line segment
$\{ (1-a,1-a,1+3a : -1/3 < a < 1 \}$ of convenient multiple points of order $n = 2$.
The convenient multiple point $p= (1,1,1)$ is strictly minimal and its critical cone is the conical hull of the vectors $\gamma_1(p)= (2,1,1)$ and $\gamma_2(p)= (1,2,1)$.

For instance, $\alpha= (3,3,2)$ is in the critical cone and applying Theorem~\ref{asymptotics} we get  
\[
	F_{r\alpha} =
	\frac{1}{\sqrt{3\pi}} 
	\left( 
	4r^{-1/2} -\frac{25}{72} r^{-3/2} + \frac{1633}{41472} r^{-5/2} 
	\right) 
	+O(r^{-7/2}),
\]
as $r \to \infty$.

Calling the one-term, two-term, and three-term truncations of this asymptotic formula $S_1(r)$, $S_2(r)$, and $S_3(r)$, respectively and comparing them with the actual values of $F_{r\alpha}$ for small $r$, we get the following table.

\begin{table}[htbp]
\footnotesize{
\begin{tabular}{|l|lllll|}
\hline 
$r$ & 1 & 2 & 4 & 8 & 16\\
\hline
$F_{r\alpha}$ & 
	0.7849731445 & 0.7005249476 & 0.5847732654 & 0.4485547669 & 0.3237528587\\
$S_1(r)$ &
 	1.302940032 & 0.9213177319 & 0.6514700159 & 0.4606588663 & 0.3257350080\\
$S_2(r)$ &
	1.189837598 & 0.8813299831 & 0.6373322117 & 0.4556603976 & 0.3239677825\\
$S_3(r)$ &
    1.202663729 & 0.8835973440 & 0.6377330283 & 0.4557312524 & 0.3239803079\\
$S_1(r)$ rel err &
	-0.6598530037 & -0.3151819005 & -0.1140557451 & -0.02698466340 & 		
	-0.006122414820 \\
$S_2(r)$ rel err &
	-0.5157685415 & -0.2580993527 & -0.08987918808 & -0.01584116640 & 
	-0.0006638514355 \\
$S_3(r)$ rel err &
    -0.5321081198 & -0.2613360125 & -0.09056461026 & -0.01599912872 & -0.0007025396085 \\
\hline
\end{tabular}
}
\vspace{1em}
\caption{Successive approximations to $p^{-r\alpha} F_{r\alpha}$ with relative errors for $\alpha= (3,3,2)$.}
\end{table}

Notice that in this case the three-term approximation to $F_{r\alpha}$ is not an improvement over the two-term approximation for $n\le 16$.
The question, which we do not discuss here, of how many terms of a divergent asymptotic series expansion to use for a given argument to obtain the best approximation/least error is called the question of `optimal truncation' or `optimal approximation'.
See \cite{PaKa2001}, for instance, for more details.
\end{example}

\begin{example}[$n < d$, repeated factors]
Consider the trivariate rational function 
\[
	F(x,y,z) = \frac{16}{(4-2x-y-z)^2(4-x-2y-z)}
\]	
in a  neighborhood $\Omega$ of the origin.
Its coefficients $F_\nu$ are all nonnegative, and its denominator $H(x,y,z)= (4-2x-y-z)^2(4-x-2y-z)$ is shown factored over $\CC[x,y,z]$.
Since $H$ contains repeated factors, we first reduce 
\[
\frac{F(x,y,z) \, dx\wedge dy\wedge dz}
{x^{\alpha_1n+1}y^{\alpha_2n+1}z^{\alpha_3n+1}}, 
\]
the differential form of the Cauchy integral of $F$, to a de Rham cohomologous form with no repeated factors, namely 
\[
\frac{[16(2\alpha_3y - \alpha_2z)n +16(2y-z)]/(yz) \, dx\wedge dy\wedge dz}{ (4-2x-y-z)(4-x-2y-z) x^{\alpha_1n+1}y^{\alpha_2n+1}z^{\alpha_3n+1}},
\]
which determines the asymptotics of $F_{r\alpha}$.
The constructive proofs of Lemma~\ref{reduce-factors} (in the case of polynomials) and Lemma~\ref{reduce-powers} to find such a cohomologous form are implemented in \texttt{amgf.sage}.

The singular variety $\var$ of this new form is the same as in the previous example and so the singularity analysis is the same. 
The convenient multiple point $p= (1,1,1)$ is strictly minimal and its critical cone is the conical hull of the vectors $\gamma_1(p)= (2,1,1)$ and $\gamma_2(p)= (1,2,1)$.

Taking $\alpha=(3,3,2)$ again, for instance, and applying Theorem~\ref{asymptotics} we get  
\[
	F_{r\alpha} = 
	\frac{1}{\sqrt{3\pi}} 
	\left( 
	4r^{1/2} +\frac{47}{72} r^{-1/2} -\frac{1967}{41472} r^{-3/2}
	\right) 
	+O(r^{-5/2}),
\]
as $r \to \infty$.

It is a coincidence that the leading coefficient above is the same as the leading coefficient in the previous example without repeated factors.
Using the denominator $(4-2x-y-z)^3(4-x-2y-z)$ instead, for instance, gives a different leading coefficient.

Calling the one-term, two-term, and three-term truncations of this asymptotic formula $S_1(r)$, $S_2(r)$, and $S_3(r)$, respectively and comparing them with the actual values of $F_{r\alpha}$ for small $r$, we get the following table.

\begin{table}[htbp]
\footnotesize{
\begin{tabular}{|l|lllll|}
\hline 
$r$ & 1 & 2 & 4 & 8 & 16 \\ 
\hline
$F_{r\alpha}$ & 
	0.9812164307 & 1.576181132 & 2.485286378 & 3.700576827 & 5.260983954 \\
$S_1(r)$ &
 	1.302940032 & 1.842635464 & 2.605880063 & 3.685270927 & 5.211760127 \\
$S_2(r)$ &
	1.515572607 & 1.992989400 & 2.712196350 & 3.760447895 & 5.264918270 \\ 
$S_3(r)$ &
    1.500123128 & 1.987527184 & 2.710265167 & 3.759765118 & 5.264676873 \\ 
$S_1(r)$ rel err &
	-0.3278824031 & -0.1690505784 & -0.04852305395 & 0.004136084917 & 
	0.009356391776 \\
$S_2(r)$ rel err &
	-0.5445854345 & -0.2644418586 & -0.09130133815 & -0.01617884746 & 	
	-0.0007478289298 \\
$S_3(r)$ rel err &
    -0.5288402039 & -0.2609763838 & -0.09052429168 & -0.01599434190 &       
    -0.0007019445473 \\
\hline
\end{tabular}
}
\vspace{1em}
\caption{Successive approximations to $p^{-r\alpha} F_{r\alpha}$ with relative errors for $\alpha= (3,3,2)$.}
\end{table}

Notice that in this case the two-term or three-term approximation to $F_{r\alpha}$ is not an improvement over the one-term approximation until somewhere between $r = 8$ and $r = 16$.
\end{example}

\begin{example}[$n \ge d$ with no repeated factors]
Consider the bivariate function
\[
F(x,y) = \frac{1}{(1 - 2x -y)(1 - x - 2y)}
\]
in a  neighborhood $\Omega$ of the origin; cf \cite[Example 4.12]{PeWi2008}.

Its coefficients $F_\nu$ are all nonnegative, and its denominator 
$H(x,y)$ factors over $\CC[x,y]$ into irreducible terms $H_1(x,y)= 1 - 2x - y$ and $H_2(x,y)= 1 - x - 2y$, both of which are globally smooth.

The set of non-smooth points of $\var= \{(x,y)\in\Omega : H(x,y) =0 \}$ is $\var' = \{ (x,y)\in\Omega : H(x,y) =\grad H(x,y) =0\}$, which consists of the convenient multiple point $p= (1/3, 1/3)$ of order $r=2$.
The point $p$ is strictly minimal and its critical cone is the conical hull of the vectors $\gamma_1(p)= (2, 1)$ and $\gamma_2(p)= (1/2, 1)$.

By Theorem~\ref{n=d}, there exists $\epsilon\in(0,1)$ such that for any $\alpha$ in this critical cone we get
\[
F_{r\alpha} 
= 
p^{-r\alpha} \left( \frac{G(p)}{p_1 p_2 |\det H'(p)|}  + O(\epsilon^r) \right) 
=
3^{(\alpha_1 + \alpha_2) n} ( 3 + O(\epsilon^r)),
\]
as $r \to \infty$.

Taking $\alpha = (4,3)$, say, letting $S(r)$ be the asymptotic expansion  above, and comparing it to the actual values of $F_{r\alpha}$ for small $r$, we get the following table.

\begin{table}[htbp]
\footnotesize{
\begin{tabular}{|l|lllll|}
\hline 
$r$ & 1 & 2 & 4 & 8 & 16 \\ 
\hline
$2187^{-r} F_{r\alpha}$ & 
	1.960219479 & 2.298399383 & 2.587511051 & 2.809909562 & 2.950100341 \\
$2187^{-r} S(r)$ &
 	3 & 3 & 3 & 3 & 3 \\
$2187^{-r} S(r)$ rel err &
	-0.5304408677 & -0.3052561804 & -0.1594153382 & -0.06765002002 & 
	-0.01691456340 \\
\hline
\end{tabular}
}
\vspace{1em}
\caption{Successive approximations to $p^{-r\alpha} F_{r\alpha}$ with relative errors for $\alpha= (4,3)$.}
\end{table}

\end{example}

\begin{example}[$n \ge d$ with repeated factors]
Consider the bivariate function
\[
F(x,y) = \frac{1}{(1 - 2x - y)^2 (1 - x - 2y)^2},
\]
which is a variation of the function of the previous example.

Since the denominator of $F$ contains repeated factors, we first reduce 
\[
\frac{F(x,y) \, dx\wedge dy}
{x^{\alpha_1 r + 1}y^{\alpha_2 r + 1}}, 
\]
the differential form of the Cauchy integral of $F$ to a de Rham cohomologous form with no repeated factors, which \texttt{amgf.sage} computes.

Reusing the analysis of the previous example and applying Theorem~\ref{n=d}, there exists $\epsilon\in(0,1)$ such that for any $\alpha$ in conical hull of the vectors $\gamma_1(p)= (2,1)$ and $\gamma_2(p)= (1/2,1)$ we get
\[
    F_{r\alpha} 
    = 
    3^{(\alpha_1 + \alpha_2) n} 
    ( 
    -3(2 \alpha_1^2 - 5 \alpha_1 \alpha_2 + 2 \alpha_2^2) r^2 - 
    3(\alpha_1 + \alpha_2) n - 9 + 
    O(\epsilon^r)
    ),
\]
as $r \to \infty$ and for some $\epsilon\in(0,1)$.

Taking $\alpha = (4,3)$, say, letting $S(r)$ be the asymptotic expansion  above, and comparing it to the actual values of $F_{r\alpha}$ for small $r$, we get the following table.

\begin{table}[htbp]
\footnotesize{
\begin{tabular}{|l|lllll|}
\hline 
$r$ & 1 & 2 & 4 & 8 & 16 \\ 
\hline
$2187^{-r} F_{r\alpha}$ & 
	30.72702332 & 111.9315678 & 442.7813138 & 1799.879232 & 7367.545085  \\
$2187^{-r} S(r)$ &
 	0 & 69 & 387 & 1743 & 7335 \\
$2187^{-r} S(r)$ rel err &
	1.000000000 & 0.3835519207 & 0.1259793763 & 0.03160169385 & 0.004417358124 \\
\hline
\end{tabular}
}
\vspace{1em}
\caption{Successive approximations to $p^{-r\alpha} F_{r\alpha}$ with relative errors for $\alpha= (4,3)$.}
\end{table}

\end{example}

\section{Remaining Proofs}\label{sec:proofs}

\begin{proof}[Proof of Lemma~\ref{residue}]
Let $f_r(w,y) = -y^{r\alpha_d -1} \Gcheck(w,y^{-1})$.  
Then for $r \ge 2$,
\begin{align*}
    R_r(w)
    =&
    \sum_{j=1}^n \lim_{y\to h_j(w)^{-1}} y^{-r\alpha_d-1} 
    ( y-h_j(w)^{-1} ) F(w,y) \\
    =&
    \sum_{j=1}^n \lim_{y\to h_j(w)^{-1}} -y^{-r\alpha_d} 
    h_j(w)^{-1} (y^{-1} -h_j(w) ) 
    \frac{\Gcheck(w,y)}{\prod_{k=1}^r ( y^{-1} -h_k(w) )}  \\
    =&
    \sum_{j=1}^n   
    \frac{f_r(w,h_j(w))}{\prod_{k\not= j} (h_j(w)-h_k(w))} \\
    =&
    \int_0^1 d\sigma_1 \int_0^{\sigma_1} d\sigma_2 \cdots \int_0^{\sigma_{r-2}}
    \left( \frac{\del}{\del y} \right)^{n-1} f_r 
    (w, (1-\sigma_1)h_1 + (\sigma_1-\sigma_2)h_2 + \cdots \\
     &
    (\sigma_{r-2}-\sigma_{n-1})h_{n-1} + \sigma_{n-1}h_n ) \; d\sigma_{n-1} \\
     & \text{(by \cite[Chapter 4, Section 7, equations (7.7) and (7.12)]
    {DeLo1993})} \\
    =& 
    \int_\Delta  
    \left( \frac{\del}{\del y} \right)^{n-1} (-1)^{n-1} 
    f_r(w, s_1h_1 +\cdots+ s_{n-1}h_{n-1} +(1-\sum_{j=1}^{n-1}s_j)h_n) \; ds \\
     & 
    \text{(by the change of variables 
    $(s_1,\ldots,s_{n-1}) = 
    (1-\sigma_1,\sigma_1-\sigma_2,\ldots,\sigma_{r-2}-\sigma_{n-1})$),}
\end{align*}
as desired.  

Notice that the $(-1)^{n-1}$ cancels with the $(-1)^{n-1}$ in the definition of $f_r$.

For $n=1$, we have 
$R_r(w)= \lim_{y\to h_0(w)^{-1}} y^{-r\alpha_d-1} (y-h_0(w)^{-1}) F(w,y) = f_r(w,h(w))$.
\end{proof}

\begin{proof}[Proof of Lemma~\ref{FL_integral}] 
First, for $n \ge 2$,
\begin{align*}
    &
    \left( \frac{\del}{\del y} \right)^{n-1} (-1)^{n-1} f(w,y) \\ 
    &=  
    \left( \frac{\del}{\del y} \right)^{n-1} 
    (-1)^n y^{r\alpha_d -1} \Gcheck(w,y^{-1}) \\
    &=
    -\sum_{j=0}^{n-1} \binom{n-1}{j} 
    \left( \frac{\del}{\del y} \right)^{n-1-j} y^{r\alpha_d -1}
    (-1)^{n-1} \left( \frac{\del}{\del y} \right)^j \Gcheck(w,y^{-1}) \\
    &=
    -\sum_{j=0}^{n-1} \binom{n-1}{j}
    (r\alpha_d - 1)^{\underline{n-1-j}} y^{r\alpha_d - n + j} 
    (-1)^{n-1} \left( \frac{\del}{\del y} \right)^j \Gcheck(w,y^{-1}) \\
    &=
    -\sum_{j=0}^{n-1} P_j(r) y^{-r\alpha_d} A_j(w,y).\\ 
\end{align*}
Thus
\begin{align*}
    &
    p^{r\alpha} (2\pi\mi)^{1-d}
    \int_{X} \frac{-R(w)}{w^{r \hat{\alpha}+1}} dw \\
    =&
    p^{r\alpha} (2\pi\mi)^{1-d} 
    \int_{X} \frac{1}{w^{r\alphahat  + 1}} 
    \int_\Delta 
    \left( \frac{\del}{\del y} \right)^{n-1} (-1)^{n-1} f(w,y) \Big|_{y=h(w,s)} ds \, dw \\
    &
    \text{(by Lemma~\ref{residue})} \\
    =&
    p^{r\alpha} (2\pi\mi)^{1-d} 
    \sum_{j=0}^{n-1} P_j(r) 
    \int_{X} \frac{1}{w^{r\alphahat  + 1}}
    \int_\Delta h(w,s)^{r\alpha_d} A_j(w,h(w,s)) ds \, dw \\
    =&
    (2\pi\mi)^{1-d} \sum_{j=0}^{n-1} P_j(r)
    \int_X \int_\Delta \frac{\phat^{r\alphahat}}{w^{r\alphahat}} 
    A_j(w,h(w,s)) (p_d h(w,s))^{r\alpha_d} ds
    \frac{dw}{\prod_{m=1}^{d-1} w_m} \\
    =&
    (2\pi)^{1-d} \sum_{j=0}^{n-1} P_j(r)
    \int_{\Xtilde} \int_\Delta 
    \prod_{m=1}^{d-1} \exp(-\mi \alpha_m n t_m) \Atilde_j(t,s)
    (p_d \htilde(t,s))^{r\alpha_d} ds \, dt \\
    & 
    \text{(via the change of variables $w = e(t)$)} \\
    =&
    (2\pi)^{1-d} \sum_{j=0}^{n-1} P_j(r)
    \int_{\Xtilde} \int_\Delta 
    \Atilde_j(t,s) \,\exp(-r\alpha_d \Phitilde(t,s)) ds \, dt,
\end{align*}
which with Lemma~\ref{integral} proves the stated formula for 
$p^{r\alpha}F_{r\alpha}$.

The formula for the case $n = 1$ follows similarly.
\end{proof}

\begin{proof}[Proof of Lemma~\ref{critical_points}]
First $\Phitilde(0,s) = 0$ and 
\[
    \Re\Phitilde(t,s) 
    = 
    -\log |p_d \htilde(t,s)| \\
    \ge
    -\log \sum_{j=1}^n s_j |p_d h_j(e(t))| \\
    >
    0
\]
for $t\not= 0$, because the sum is convex 
and $|h_j(w)^{-1}|>|p_d|$ for $w\not= \phat$ since $p$ is strictly minimal.

Now by the calculation in the proof of Proposition~\ref{Hessian}, for all $l < d$ we have
\[
    \del_l \Phitilde(\theta^*) 
    = 
    -\mi \sum_{j=1}^{r} s_j^* 
    \frac{p_l \del_l H_j(p)}{p_d \del_d H_j(p)} 
    + \mi\frac{\alpha_l}{\alpha_d} 
    =
    0,
\]
where the last equality holds since $p$ is critical for $\alpha$.
Also $\del_l \Phitilde(\theta^*)=0$ for $d \le l \le n + d - 2$ since 
$\Phitilde(0,s)$ is constant.  
Thus $\grad \Phitilde(\theta^*)= 0$.
Now $\det\Phitilde''(\theta^*) \not= 0$, since $p$ is nondegenerate for $\alpha$.
So 
there is a neighborhood of $\theta^*$ in which
$\theta^*$ is the only zero of $\Phitilde'$.
Thus, shrinking $\Xtilde\times\Delta$ if needed, $\theta^*$ is the unique stationary point of $\Phitilde$.
\end{proof}

%
\section{Acknowledgements}

We would like to thank the anonymous referee for her/his thorough and constructive comments on our first draft.
Our presentation is much more user-friendly because of them.
\bibliographystyle{amsalpha}
\bibliography{combinatorics}
\end{document}